\documentclass[12pt,a4paper]{amsart}
\usepackage{amsmath}
\usepackage{amssymb}
\newtheorem{theorem}{Theorem}[section]
\newtheorem{oldtheorem}{Theorem}

\newtheorem{lemma}[theorem]{Lemma}
\newtheorem{corollary}[theorem]{Corollary}

\theoremstyle{statement}
\newtheorem{remark}[theorem]{Remark}
\newtheorem{definition}[theorem]{Definition}

\theoremstyle{fancyproclaim}

\numberwithin{equation}{section}
\renewcommand{\theenumi}{\roman{enumi}}
\newcommand{\UB}{\textup{(UB)}}
\newcommand{\efs}{eigenfunctions }
\newcommand{\be}{\begin{equation}}
\newcommand{\ee}{\end{equation}}
\newcommand{\ben}{\begin{equation*}}
\newcommand{\een}{\end{equation*}}
\newcommand{\lb}{\label}
\newcommand{\ol}{\overline}

\DeclareMathOperator{\tr}{tr}
\DeclareMathOperator{\dist}{dist}

\newcommand{\im}{\mathop{\mathrm{Im}}}
\newcommand{\re}{\mathop{\mathrm{Re}}}

\providecommand{\abs}[1]{\left\lvert#1\right\rvert}
\providecommand{\norm}[1]{\left\lVert#1\right\rVert}

\newcommand{\Cset}{\mathbb{C}}
\newcommand{\Rset}{\mathbb{R}}

\usepackage{mathrsfs}  

\newcommand{\calD}{{\mathcal D}}
\newcommand{\calK}{{\mathcal K}}
\newcommand{\calL}{{\mathscr L}} 

\newcommand{\hilb}{\mathfrak{H}}

\newcommand{\xiexp}{\varSigma^{(exp)}}
\newcommand{\lexp}{\Lambda^{(exp)}}

\newcommand{\gubspace}{A^2_a(\mathbb{R},w^{-2})}
\newcommand{\Ksigma}{\calK_\varSigma}
\newcommand{\KLambda}{\calK_\Lambda}
\newcommand{\qexp}{quasi-exponential}
\newcommand{\bexp}{$B$-\qexp}
\newcommand{\rreg}{\emph{right-re\-gu\-lar}}
\begin{document}
\date{\today}
\title[$B$-quasi-exponentials]{On unconditional basisness of
$B$-quasi-exponentials}
\author[Arkadi Minkin]{Arkadi Minkin}

\address{Parametric Technology Ltd., MATAM, 31905 Haifa, Israel}
\email{arkadi\_minkin@yahoo.com}
\begin{abstract}
We consider operator $A=K^{-1}, Kh=Bh+(h,f)g$ in separable Hilbert space $\hilb$ where $B$ is quasinilpotent operator whose resolvent is entire operator function of exponential type $a$.
Let $\varphi(z)$ be Fredholm determinant of $K$.
We prove two necessary conditions for unconditional basisness of eigenfunctions of $A$
without any a priori restrictions on its spectrum $\Lambda$:
weights
$w^2(\lambda)$ and
$W^2(\lambda)$ on $\Rset$ must satisfy Muckenhoupt condition ($A_2$).

Here $w^2(\lambda) := \norm{g(\lambda)}^2$ and
$W^2(\lambda) := \abs{\varphi(\lambda)}^2/\left(\delta(\lambda)\cdot w^2(\lambda)  \right)$ where vec\-tor-valued function $g(z) := (I-zB)^{-1}g$ and
$\delta(z)$ is regularized distance from point
$z\in\Cset$ to $\Lambda$.

\end{abstract}


\maketitle

\section*{INTRODUCTION\label{sec:Intro}}
One of the main challenges of spectral theory is question of similarity to normal for non-selfadjoint operators. If operator $A$ has compact resolvent
the problem reduces to
 \emph{unconditional basisness} of its \efs and we will write $A \in\UB$. A celebrated example
 is question of unconditional basisness of exponentials or more generally of values of reproducing kernels deeply investigated in the classical papers
 of B.S.Pavlov, S.V.Hru\v{s}c\"ev and N.K.Nikol'skii
 \cite{Pav79}, \cite{Hru79},\cite{Nik80}, giving birth to \emph{projection} method,
 see also their excellent survey article \cite{HNP}.


Similarity to normal problem has been attacked from various directions,
for dissipative as well as for non-dissipative operators,
see for example
works of V.E.Katsnel'son, B.S.Pav\-lov, N.K.Ni\-kol'\-skii, V.I.Va\-sy\-nin,
S.Treil, S.Ku\-pin \cite{Kats67}, \cite{Pav73},\cite{NikVs},\cite{NikTr},\cite{KuTr}
to name a few.
We refer the reader to \cite{Nik} and \cite{NikVs} for more details.

 Recently G.~M.~Gub\-reev substantially developed \textit{projection} method
in a series of remarkable publications
\cite{gub99}--\cite{gub-izv}
and others. He introduced classes $\xiexp$
and $\lexp$ of quasinilpotent operators
with exponential resolvent's growth and considered
their one- and
finite-dimensional perturbations continuing A.P.Khromov's investigations of
perturbations of Volterra operators
\cite{Khr}.

However G.M.Gubreev imposed several
a priori restrictions on the operators in question.
Our goal is to deduce basic necessary conditions for unconditional basisness of \efs
\emph{without any a priori hypotheses}.
In order to describe existing results in more details we
will introduce some notations
sticking to those from \cite{gub00} and \cite{gub03}.

Let $\xiexp$ denote the set of all bounded linear operators $B$
in a separable Hilbert space $\hilb$, satisfying the following conditions:
\begin{enumerate}
\item spectrum $\sigma(B) = \{0\}, \ \ker{B} = \{ 0\}$;
      \lb{en-spB}
\item $\ker{B}^* = \{ 0 \}$;
      \lb{en-ker}
\item $(I-zB)^{-1}$ is of finite exponential type $a>0$;
      \lb{en-exp-type}
\item the semigroup of operators $V(t):=\exp(-\mathrm{i}B^{-1}t), t\ge 0$
      is of class $C_0$.
      \lb{en-C0}
\end{enumerate}
Recall that the latter means that $V(t)$ is continuous in strong operator topology for $t\ge0$ and
\(
   \lim_{t\rightarrow +0}V(t)h = h, \; \forall h\in\hilb.
\)
Obviously such $B$ is quasinilpotent due to condition \eqref{en-exp-type}.

Let $\lexp$ be the set of all bounded linear operators $B$ in $\hilb$, satisfying conditions \eqref{en-spB},\eqref{en-exp-type} above, whereas
conditions \eqref{en-ker} and \eqref{en-C0} are replaced by dissipativity of $B$
\begin{enumerate}
\item[(\ref{en-ker}$^\prime$)]  $(1/2\mathrm{i})(B-B^*) \ge 0$.
\end{enumerate}
Note that for dissipative $B$ the subspace $\ker{B}$ reduces it
whence automatically $\ker{B^*} = \ker{B} = \{ 0 \}$
for $B\in\lexp$ and \eqref{en-ker} is valid too. So $\lexp\subset\xiexp$.
Moreover for any automorphism $S$ of $\hilb$ we have
\ben
   S\lexp S^{-1} := \{ SBS^{-1}, \; B\in  \lexp\} \subset\xiexp.
\een
\begin{definition}
A vector-valued function of the form
\ben
   g(z) = g(z,B) := (I-zB)^{-1}g, \; B\in\xiexp, \; g\in\hilb,
\een
is called a \bexp{ } or merely \qexp{ } if the operator is known from the context.
\end{definition}
\begin{definition}[Gubreev]
\bexp{ } $g(\cdot)$ is said to be
\rreg{ }
if there exists constant $M >0$ such that
\begin{equation*}
   \int_\Rset \abs{(g(x),h)}^2\norm{g(x)}^{-2}\mathrm{d}x
   \le M\norm{h}^2, \quad \forall h\in\hilb
\end{equation*}
and is said to be \emph{regular}
if there exist constants $m,M >0$ such that
\begin{equation*}
   m\norm{h}^2 \le
   \int_\Rset \abs{(g(x),h)}^2\cdot\norm{g(x)}^{-2}\mathrm{d}x
   \le M\norm{h}^2, \quad \forall h\in\hilb.
\end{equation*}
\end{definition}
\begin{definition}[Gubreev]
A pair of vectors $f,g\in\hilb$ is said to be \emph{compatible}
with an operator $B$ of class $\xiexp$ if $\ker{K} = \ker{K^*} = \{0\}$
for operator
\begin{equation}
  \label{eq-op-K}
    Kh := Bh + (h,f)g, \quad h \in \hilb.
\end{equation}
In this case we will consider a closed densely defined operator $A:=K^{-1}$.
Its action may be also described as follows:
\begin{equation}
   \label{eq-op-A}
   \text{\textit{if} } h=l+(B^{-1}l,f)g \quad \text{\textit{then} }\ Ah = B^{-1}l,
   \quad
   l\in\mathfrak{L}:=B\hilb.
\end{equation}
The set of such operators $A$ will be denoted $\Ksigma$.
In case we take $B$ from $\lexp$ we will denote $\KLambda$ the set of respective operators $A$.
\end{definition}

For $B\in\xiexp$ introduce another \qexp
\ben
   f_*(z) := f(z,B_*) = (I-zB_*)^{-1}f, \; B_* := -B^*.
\een
\begin{remark}
  \lb{rem:adjoint}
Note that $B_* \in\xiexp$ because the semigroup
\ben
   V(t)^* \equiv\exp(-\mathrm{i}B_*^{-1}t)
\een
is also of class $C_0$, see Corollary 1.3.2 in
\cite{Neer}.
Therefore $A\in\Ksigma$ implies $A_*:=-A^*\in\Ksigma$.
\end{remark}

To avoid unnecessary complications throughout the paper we will always assume
\emph{simplicity of eigenvalues of $A$}.
Set $z_*:=-\bar{z}$.
Then the resolvent of $A\in\Ksigma$ takes the form
\begin{align}
 \lb{eq-res-A}
(A-zI)^{-1}h & =
   B(I-zB)^{-1}h + \calL(z){h}\\
   \calL(z){h} &:= \varphi^{-1}(z)\cdot\left(h,f_*(z_*)\right)g(z)\notag \\
   \varphi(z) &:= 1-z(g(z),f). \notag
\end{align}
Denote $\Lambda = \{ \lambda_k\}$ the spectrum
of $A$ enumerated in the order of nondecreasing absolute values.
Of course $\Lambda = \varphi^{-1}(0)$.
A.~P.~Khromov pointed out to us that
$\varphi(z)$ is Fredholm determinant \cite[\S6.5.2]{Pit} of $K=A^{-1}$.
Note also that
\efs of $A$ and $A^*$ coincide respectively with vectors
\begin{equation}
  \label{eq-eigen}
  \{ g(\lambda_k) \}_{\lambda_k\in\Lambda}
\end{equation}
and
\begin{equation}
 \label{eq-conj-eigen}
   \{ f_*(\lambda_{k*}) \}_{\lambda_k\in\Lambda}.
\end{equation}
In the sequel we will frequently use the linear resolvent growth condition:
\ben
  \lb{LRG}
   \norm{R_z(A)} \asymp d(z), \quad d(z) := \dist(z,\Lambda) \tag{LRG}
\een
which stems from $A\in\UB$. Indeed, similarity of $A$ to normal operator $N$
means that
$R_z(A) = SR_z(N)S^{-1}$ for some bounded and boundedly invertible operator $S$.
Then $\norm{R_z(A)} \le \norm{S}\cdot\norm{S^{-1}}\cdot\norm{R_z(N)}$.
Interchanging $A$ and $N$ we get
inverse inequality and thus arrive to \eqref{LRG} because $\norm{R_z(N)} = d(z)$.

Note also a helpful formula
\begin{equation}
\label{eq-resB}
-\mathrm{i}B(I-zB)^{-1}h = \int_0^a \mathrm{e}^{\mathrm{i}zt}V(t)h\mathrm{d}t, \; h\in\hilb, \ z\in\Cset.
\end{equation}
whence for any fixed $c>0$
\be
  \lb{eq-res-norm}
  \norm{R_z(B)} \le M/(1+\abs{\im{z}}),\quad \im{z}\ge -c,\ M=M(c).
\ee
Therefore $(I-zB)^{-1}$ has exponential type $a>0$
in the lower half-plane and is bounded in the upper half-plane.
In addition mention a helpful formula
\begin{equation}
\label{eq-resB-est}
\int_\Rset \norm{R_\lambda(B)h}^2\mathrm{d}\lambda
  \le
  \frac1{2\pi}\int_0^a \norm{V(t)h}^2\mathrm{d}t \le K\norm{h}^2,
\end{equation}
where $K=\sup_{0\le t \le a} \norm{V(t)}^2/(2\pi)$.

\section{GUBREEV'S RESULTS.}\label{sec:gub}
G.M.Gubreev invented an integral transformation related to each
Muckenhoupt weight $w^2(\lambda)$ on $\Rset$.
Recall that this condition for a positive weight
$v(\lambda),\ {\lambda\in\Rset}$ reads as follows,
see Lemma 4.2 in  \cite[Part III]{HNP}:
\ben
   \lb{Muck}
    \sup_{z\in\Cset_+}v(z)\cdot(v^{-1})(z) < \infty \tag{$A_2$}.
\een
Here $\Cset_+$ is the upper half-plane and
$v(z)$ stands for harmonic continuation of $v(\lambda)$ to $z\in\Cset_+$.
For a Muckenhoupt weight $w^2$ G.M.Gubreev \cite[\S2]{gub00} defined function
$y_w(t)$ on $\Rset_+$
and introduced transformation
\ben
 \lb{G-trf}
  (\calD_w f)(u) := \frac1{\sqrt{2\pi}} \int_0^\infty y_w(u,t)f(t)\mathrm{d}t.
\een
It maps $L^2(\Rset_+)$ isomorphically onto weighted Hardy space $H^2_+(w^{-2})$
\cite{gub00}
and is bounded from below.
Notice also that Gubreev's transformation $\calD_w$ maps $L^2(0,a)$ isomorphically onto
Hilbert space $\gubspace$ \cite{gub00} of entire functions $F(z)$ of exponential type
such that
\begin{itemize}
	\item $h_F(\pi/2) \le 0,\ h_F(-\pi/2) \le a$;
	\item ($F,G)_{\gubspace} :=
	\int_\Rset F(\lambda)\cdot\ol{G(\lambda)}\cdot{w}^{-2}(\lambda)\mathrm{d}\lambda,
	\quad F,G \in \gubspace.$
\end{itemize}
Recall that for entire function $F(z)$ of exponential type its indicator function is
\ben
     h_F(\theta) := \limsup_{r\rightarrow\infty} r^{-1}\log\abs{F(r\mathrm{e}^{\mathrm{i}\theta})},
                      \; -\pi < \theta \le \pi.
\een
Denote $J_a$ the integration operator in $L^2(0,a)$:
\ben
   (J_af)(x) := \mathrm{i}\int_0^x f(s)\mathrm{d}s,\quad f \in L^2(0,a)
\een
and introduce $w$-\qexp{}
\ben
    y_w^a(z,t) := \left((I-z J_a)^{-1}y_w\right)(t),
    \quad 0 \le t\le a; \; z\in\Cset.
\een
With an operator $A\in K_\varSigma$ we associate two weights on the real line:
\begin{eqnarray}
   w^2(x)    &:= \norm{g(x)}^2\phantom{{_*}{_*}}
                         &= \norm{(I-xB)^{-1}g}^2, \label{eq-w}   \\
   w^2_*(x)  &:= \norm{f_*(x_*)}^2
                         &= \norm{(I+xB^*)^{-1}f}^2. \label{eq-w-star}
\end{eqnarray}
Let us agree that throughout the paper all operations on sets  are understood in element-wise way.
Now we are in position to formulate Gubreev's results.
\begin{oldtheorem}[Gubreev, theorem 0.2 in \cite{gub03}]
 \lb{thm:Gub-A}
 Let $A\in\Ksigma$ and assume that
\be
   \lb{eq-semi-bnd}
   \inf\im\Lambda > 0.
\ee
Then $A\in\UB$ iff the following conditions are satisfied:
\begin{align}
	\lb{en-wA2}         w^2(x)           &\in\eqref{Muck}; \tag{G1}\\
	\lb{en-WA2}         \abs{\varphi(\lambda)}^2/w(\lambda)^2
	                    &\in\eqref{Muck}; \tag{G2}\\
	\lb{en-Biso}        B\ \text{is isomorphic to } J_a: B &= SJ_aS^{-1};\tag{G3}\\
	\lb{en-on_gf}       Sy_w^a           &= g;\tag{G4}\\
	\lb{en-ind-phi}     h_\varphi(\pi/2) &= 0,\; h_\varphi(-\pi/2) = a;\tag{G5}\\
	\lb{eq-Lambda-carl} \Lambda          &\in(C);\tag{G6}
\end{align}
where $S$ is an isomorphism from $L^2(0,a)$ onto $\hilb$.
\end{oldtheorem}
Recall that for a sequence of distinct complex numbers $\{\mu_k\}_{k=-\infty}^{\infty}$
in $\Cset_+$ (or in $\Cset_-$) the Carleson condition (C) means \cite[Ch. 7]{Gar}
\ben
    \inf_k \prod_{j\ne k} \abs{\frac{\mu_k -\mu_j}{\mu_k -\ol{\mu_j}}} > 0.
\een
This theorem remains also valid for the case of multiple spectrum provided that
eigenvalues multiplicities are uniformly bounded.
Of course condition \eqref{eq-semi-bnd} may be replaced by spectrum semi-boundedness
$\inf\im\Lambda>-\infty$ with appropriate change in formulations
but still it is too restrictive.

\begin{remark}
Note that important estimate (4.12)
of $\calL(z)$ from below in \cite{gub00} uses multiplicative representation of $\varphi(z)$
(same as ours \eqref{eq-phi-mult})
and  unnamed formula after it:
\ben
   \norm{\calL(z)} \ge \im\frac{\varphi^\prime(z)}{\varphi(z)} > d
\een
where $z$ lies in a certain half-plane.
However the second inequality above is correct \emph{only for semi-bounded} spectrum.
\end{remark}
In subsequent article \cite{gub03} G.M.Gub\-reev offered important refinement
of theorem \ref{thm:Gub-A} by
replacing  \eqref{eq-semi-bnd} with requirement that there is a strip free of spectrum:
\ben
   \lb{eq-free-strip}
   \boxed{
   \dist(\Lambda,\Rset) > 0.
   }
   \quad \tag{$G0$}
\een
Let $\Lambda_\pm := \Lambda \cap \Cset_\pm$.
We will need a counterpart of \eqref{eq-Lambda-carl}:
\ben
   \Lambda_\pm\in(C). \tag{$G6^\prime$}\lb{eq-Lambda-pm-carl}
\een
\begin{oldtheorem}[Gubreev, theorem 0.3 in \cite{gub03}]
 \lb{thm:Gub-B}
Let $A\in\KLambda$. Assume that $g(\cdot)$ is \rreg{} and the
spectrum $\Lambda$ satisfies \eqref{eq-free-strip}.
Then $A\in\UB$ iff conditions
\eqref{en-wA2}--\eqref{eq-Lambda-pm-carl} are valid where
\eqref{eq-Lambda-pm-carl} replaces \eqref{eq-Lambda-carl}.
\end{oldtheorem}
Thus spectrum restrictions were \emph{nearly removed} at expense of imposing
a priori \emph{right-regularity} of $g(\cdot)$.
This condition seems to be difficult to check unless we consider
basisness of $w$-\qexp{s}
\cite{gub99}
or \emph{regular}
de Branges space ${\mathcal H}(E)$ \cite{gub01} where this property is actually built in definitions.

\begin{remark}
Recall that these general results of G.M.Gubreev contain many earlier achievements of other authors. Indeed, if we choose $\hilb=L^2(0,a),\ B=J_a$, take
$g=y_w^a$ for some Muckenhoupt weight $w^2$ and require $\Lambda\subset\Cset_+$,
then theorem \ref{thm:Gub-A} gives criterion of unconditional basisness of
$w$-\qexp{s}, equivalent to theorem 1
in \cite{Hru87} for $\theta(z) \equiv \exp(\mathrm{i}az)$.

For  $w(\lambda)\equiv1$ $w$-\qexp{s} become usual exponentials. Thus
theorem \ref{thm:Gub-A} reduces to
 B.S.Pavlov's result \cite{Pav79} whereas theorem \ref{thm:Gub-B} will coincide with our main theorem \cite{Min92} though the case of unconditional basisness in the span
in the latter  article remained uncovered.
\end{remark}

\begin{remark}
Note that full criterion of
unconditional basisness of exponentials in $L^2(0,a)$ \emph{without} spectrum restriction
\cite{Min91}
does not follow from theorem \ref{thm:Gub-B}.
Nevertheless, applying his approach to operators in de Branges space ${\mathcal H}(E)$, G.M.Gubreev announced important results concerning basisness of their reproducing kernel bases:
theorem 8 in \cite{gubtar03} (sufficient conditions) and theorem 1 in \cite{gubtar06}
(criterion without spectrum restrictions).

The latter result supersedes
that of \cite{Min91} and partially extends N.K.Ni\-kol\-skii's criterion
\cite{Nik80} to
the case of arbitrary spectrum $\Lambda\subset\Cset$. Indeed, de Branges space
${\mathcal H}(E)$ is isometrically isomorphic to the special model subspace
$K_\Theta $ in $H^2_+$, namely $\Theta(z) := \ol{E(\bar{z})}/E(z)$ -- inner function meromorphic in the whole plane.
\end{remark}

\begin{remark}
Recall that the main tool in \cite{Pav79} was boundedness of certain skew projector built via \emph{generating} function of the system of exponentials.
 This function was successfully applied to the question of basisness of scalar and vector exponentials, see e.g. \cite{AvdIv},\cite{Le}.

But if $A\in\Ksigma$ it must be some operator-valued entire function
$M(z)$ such that $M(\lambda_k)g(\lambda_k) = 0, \ \lambda_k\in\Lambda$.
However, to the best of our knowledge neither its existence nor criterion of boundedness
of Hilbert transform with \emph{operator} weight were established to date.
\end{remark}

\section{MAIN RESULT}
In the present paper
we establish necessity of
\eqref{en-wA2},\eqref{en-WA2} (the latter condition is suitably changed)
\emph{without} \eqref{eq-free-strip} and \emph{right-regularity} of $g(\cdot)$
which sets foundation to obtain full criterion for $A\in\UB$
without any a priori hypotheses.

Introduce regularized distance
\be
  \label{eq-dist-fn}
  \delta(z) := d(z)/(1+d(z)), \quad z \in \Cset
\ee
and define weight
\ben
  \boxed{
          W(\lambda) := \frac{\abs{\varphi(\lambda)}}{w(\lambda)\cdot \delta(\lambda)}
        }, \quad \lambda\in\Rset.
\een
\begin{theorem}
  \lb{thm:neces}
  If $A\in\Ksigma$  then
the following statements hold true.

\begin{enumerate}
	\item $A\in\UB \Longrightarrow g(\cdot)\ \text{is \rreg}\
	       \Longrightarrow w^2\in\eqref{Muck}$;	
	      \lb{en-cond1}
  \item $A\in\UB \Longrightarrow f_*(\cdot)\ \text{is \rreg}\
         \Longrightarrow\ w_*^2 \in\eqref{Muck}$;
	      \lb{en-cond2}
	\item $A\in\UB \Longrightarrow W^2\in\eqref{Muck}$ .
	      \lb{en-cond3}
\end{enumerate}
\end{theorem}

Recall that we don't restrict zeros of $\varphi(z)$ besides that they are simple.
However,
without loss of generality we assume that there are no
real ones.

Throughout the paper we will use the following notations.
\begin{itemize}
\item $c,C,\varepsilon$  stand
   for different positive constants which may vary even during a single computation;
\item $x \lesssim y$ means that $x \le C y$,
       constant $C$ doesn't depend on indeterminates $x,y$
       and $x \asymp y$ is equivalent to
       $x \lesssim y$ and $y \lesssim x$.
\end{itemize}
Let us describe an outline of the paper.
In section \ref{sec:res}, lemma \ref{cor-M-est},
we give a double-sided estimate for $\calL(z)$.
Emphasize that its predecessor lemma \ref{lem:est-M-below}
is the key element in the whole proof. It provides a lower bound for $\calL(z)$
which is unattainable by methods from \cite{gub00},\cite{gub03} if spectrum $\Lambda$ is arbitrary.
Then in section \ref{sec:wght} we establish weighted estimates of the integral
over $\Rset$ of $R_\lambda(A),\ \calL(\lambda)h$ and $\calL^*(\lambda)h$.
These estimates are obtained by refining G.M.Gubreev's reasoning, namely we split $\Rset$ into several subsets and argue differently for each of them.
In the last section \ref{sec:proof} our main result -- theorem \ref{thm:neces} is proved.

\section{ESTIMATE OF FINITE-MEROMORPHIC RESOLVENT'S PART $\calL(z)$\label{sec:res}}
A more precise formula for $\varphi(z)$ reads as follows
\be
   \lb{eq-phi-detail}
   \varphi(z) = 1 - z(g,f) +\mathrm{i}z^2\int_0^a \mathrm{e}^{\mathrm{i}zt}(V(t)g,f)\mathrm{d}t,
\ee
Thus $\varphi(z)$ belongs to the class Cartwright,
assumes multiplicative representation
\be
  \lb{eq-phi-mult}
   \varphi(z) =
   \mathrm{e}^{\mathrm{i}dz} \prod\nolimits^\prime_{\lambda_k \in \Lambda} (1-z/\lambda_k)
\ee
and
\ben
   h_\varphi(+\frac{\pi}{2}) = -{d} +\frac{l}{2} \le 0,\quad
   h_\varphi(-\frac{\pi}{2}) =  {d} +\frac{l}{2} \le {a},
\een
whence ${d} \ge l/2$.
Here $l$ is the \emph{width} of the indicator diagram
of $\varphi(z)$,
i.e length $\abs{I_\varphi}$ of the interval
$I_\varphi\subset\mathrm{i}\Rset$
\cite[\S17.2, p.127]{Le}.
\begin{lemma}
  \lb{lem:b-posit}
  If $A\in\Ksigma$ and has infinite discrete spectrum then $l>0$
  and therefore ${d}>0$ too.
\end{lemma}
\begin{proof}
Assume on the contrary that $l=0$.
Take two eigenvalues $\lambda_{1,2}$ and define function
\ben
    \psi(z) := \frac{\varphi(z)}{(z-\lambda_1)(z-\lambda_2)}.
\een
Obviously $\abs{I_\psi}=0$. From \eqref{eq-phi-detail} it follows that
$\psi(x) \in L^2(\Rset)$. Then by
Paley-Wiener theorem $\psi(z)\equiv 0$ and so is $\varphi(z)$.
\end{proof}

Following \cite[p.908]{gub00} observe that
\begin{equation}
  \label{eq-biort-0}
  (g(\lambda),f_*(\mu_*)) = (\varphi(\lambda)-\varphi(\mu))/(\mu-\lambda)
\end{equation}
whence
\begin{equation}
  \label{eq-biort-1}
  (g(\lambda_k),f_*(\mu_*)) = -\varphi(\mu)/(\mu-\lambda_k),\;
  \lambda_k \in \Lambda
\end{equation}
and
\begin{equation}
 \label{eq-biort-2}
 (g(\lambda_k),f_*(\lambda_{j*})) = -\delta_{jk}\cdot \varphi^\prime{}(\lambda_k).
\end{equation}

Introduce functions
\ben
    \varphi_k(\lambda) :=
    \frac{\varphi(\lambda)}{(\lambda-\lambda_k)\varphi^\prime(\lambda_k)}.
\een
In the formulae \eqref{eq-f-star-expand}-\eqref{eq-g-norm} below we assume that $A\in\UB$.
Expand $f_*(\mu_*)$ over \efs \eqref{eq-conj-eigen} of $A^*$:
\begin{equation}
  \label{eq-f-star-expand}
  f_*(\mu_*) = \sum_j \overline{\varphi_j(\mu)}
  \cdot
  f_*(\lambda_{j*}).
\end{equation}
Then
\begin{equation}
  \label{eq-f-star-norm}
  \norm{f_*(\mu_*)}^2 \asymp \sum_j
   \abs{\varphi_j(\mu)}^2 \norm{f_*(\lambda_{j*})}^2.
\end{equation}
Quite analogously, expanding $g(\lambda)$ over \eqref{eq-eigen},
\begin{equation}
  \label{eq-g-expand}
  g(\lambda) = \sum_k \varphi_k(\lambda)g(\lambda_k)
\end{equation}
we get
\begin{equation}
  \lb{eq-g-norm}
  \norm{g(\lambda)}^2 \asymp
     \sum_k \abs{\varphi_k(\lambda)}^2
   \norm{g(\lambda_k)}^2.
\end{equation}

From $A\in\UB$ follows uniform minimality of \efs, i.e.
\begin{equation}
 \label{eq-UM}
\norm{f_*(\lambda_{k*})}\norm{g(\lambda_k)} \tag{UM}
\asymp
\left| \varphi^\prime{}(\lambda_k) \right|.
\end{equation}

\begin{lemma}
  \label{lem:est-M-below}
  If $A\in\UB$, then
\(
   \norm{\calL(z)} \ge C > 0, \quad z\in\Cset.
\)
\end{lemma}
\begin{proof}
Observe that
\ben
   -\tr{\calL(z)} = \frac{\varphi^\prime(z)}{\varphi(z)}
             = \sum\nolimits^\prime \frac1{z-\lambda_k} + \mathrm{i}{d}.
\]
Applying Cauchy-Schwarz inequality we get
\begin{align}
   &\abs{\sum\nolimits^\prime \frac1{z-\lambda_k}}
   \equiv
   \abs{\sum\nolimits^\prime \frac{\varphi_k(z)}{\varphi(z)} \cdot
   \varphi^\prime(\lambda_k)} \notag\\
   &\lesssim
    \frac1{\abs{\varphi(z)}}
      \left[\sum \abs{\varphi_k(z)}^2 \norm{g(\lambda_k)}^2
      \right]^{1/2}
      \cdot
      \left[
      \sum\abs{\varphi_k(z)}^2\norm{f_*(\lambda_{k*})}^2
      \right]^{1/2}
      \notag\\
   &\asymp
   \frac{\norm{g(z)} \cdot    \norm{f_*({z_*})}}{\abs{\varphi(z)}}
   \asymp \norm{\calL(z)},\notag
   \lb{eq-sum-rat}\\
\end{align}
where we used \eqref{eq-UM} and norm estimates \eqref{eq-f-star-norm} and \eqref{eq-g-norm}.
From \eqref{eq-sum-rat} and lemma \ref{lem:b-posit} stems
\ben
  0 < {d} \le \abs{\tr{\calL(z)}} + C\norm{\calL(z)} \lesssim \norm{\calL(z)},
\]
and the lemma is proved.
\end{proof}

Next let us establish double-sided estimate of $\calL(z)$.
\begin{lemma}
\label{cor-M-est}
If $A\in\UB$, then
\be
  \label{eq-est-M}
  \boxed{
          \norm{\calL(z)} = \frac{\norm{g(z)}\cdot\norm{f_*(z_*)}}{\abs{\varphi(z)}}
          \asymp \frac1{\delta(z)}},
  \quad \abs{\im{z}}  \le c, \; c >0.
\ee
\end{lemma}
\begin{proof}
The upper estimate stems directly from relation \eqref{LRG} and
estimate \eqref{eq-res-norm} of the Fredholm resolvent of operator $B$.

Let us turn to the estimate from below.
Due to lemma \ref{lem:est-M-below} it is enough to proof
it in small circles of radii $\varepsilon$ around eigenvalues
intersecting with the horizontal strip
$\abs{\im{z}} \le c$.
However in these circles the resolvent $R_z(B)$
is uniformly bounded from above,
see \eqref{eq-res-norm}.
 Together with \eqref{LRG} it yields
lemma's assertion.
\end{proof}
\begin{corollary}
If $A\in\UB$, then
\begin{equation}
  \label{eq-est-ginv}
   \frac1{\norm{g(\lambda)}}
   \asymp
   \frac{\delta(\lambda)\cdot\norm{f_*(\lambda_*)}}{\abs{\varphi(\lambda)}},
   \; \lambda \in \Rset.
\end{equation}
\end{corollary}

\section{WEIGHTED ESTIMATE OF RESOLVENT'S INTEGRAL}\label{sec:wght}
\begin{lemma}
  \lb{lem:estA}
If $A\in\UB$ then
\begin{equation}
  \label{eq-estA}
   \sup_{\lambda_k\in\Lambda}\int_\Rset   \frac{\delta^2(\lambda)}{\abs{\lambda_k-\lambda}^2} \mathrm{d}\lambda
   < \infty.
\end{equation}
\end{lemma}
\begin{proof}
Recall that always $\delta(\lambda) <1$ and
consider two cases:
\begin{enumerate}
\item \label{ge1} $\abs{\im\lambda_k} \ge 1$ ;
\item \label{lt1} $\abs{\im\lambda_k} < 1$.
\end{enumerate}
\noindent\emph{Case ({\ref{ge1}}).}
Remove $\delta^2(\lambda)$ from \eqref{eq-estA}
and conclude that the integral $\le \pi$.

\noindent\emph{Case ({\ref{lt1}}).}
We can assume for deficiency that
$0<\im\lambda_k < 1$. Fix $k$ and divide $\Rset$ into two sets:
\ben
   \Xi_0 = \{ \abs{\lambda-\re\lambda_k}\ge1\}, \;
   \Xi_1 = \Rset \setminus \Xi_0.
\]
Taking an integral in \eqref{eq-estA} over $\Xi_0$ and replacing there $\delta(\lambda)$ by $1$ we conclude that it is bounded by an absolute constant.
Further, take the same integral over $\Xi_1$ and use trivial bound
$\delta(\lambda) < d(\lambda)$. Set $x=\re\lambda_k,\; y=\im\lambda_k >0$.
Then we need to estimate an integral
\ben
   \int_{\abs{\lambda-x}\le 1}
   \frac{d(\lambda)^2}{(\lambda-x)^2+y^2}\mathrm{d}\lambda
\]
which is $\le 2$ because
$d(\lambda) \le \dist(\lambda,\lambda_k)
= \sqrt{(\lambda-x)^2+y^2}$.
\end{proof}

\begin{lemma}
If $A\in\UB$, then
\begin{equation}
 \label{eq-res-est}
\int_\Rset \norm{(A-\lambda)^{-1}}^2 \cdot \delta(\lambda)^2\mathrm{d}\lambda
\lesssim\norm{h}^2.
\end{equation}
\end{lemma}
\begin{proof}
Expand $h$ into \efs
\begin{equation}
   \label{eq-expan-g}
   h = \sum_k h_k g(\lambda_k) /\norm{g(\lambda_k)}.
\end{equation}
Then
\ben
   (A-\lambda)^{-1}h = \sum_k h_k \frac{g(\lambda_k)}{\lambda_k - \lambda}
   \cdot \frac{1}{\norm{g(\lambda_k)}}
\]
and
\ben
   \norm{(A-\lambda)^{-1}h}^2
   \asymp
   \sum_k \abs{h_k}^2 \frac1{\abs{\lambda_k-\lambda}^2}.
\]
Therefore we have
\begin{equation}
  \label{eq-res-asymp}
   \int_{\Rset}\norm{(A-\lambda)^{-1}h}^2
   \cdot \delta^2(\lambda)\mathrm{d}\lambda
   \asymp \sum_k \abs{h_k}^2
   \int_\Rset \frac{\delta^2(\lambda)}{\abs{\lambda_k-\lambda}^2}\mathrm{d}\lambda
\end{equation}
and it suffices to apply inequality \eqref{lem:estA}.
\end{proof}

\begin{lemma}
\label{lem-res-est-below}
If $A\in\UB$ then
the following estimates hold true
\begin{equation}
\label{eq-est-main1}
\int_\Rset \delta^2(\lambda) \frac1{\abs{\varphi(\lambda)}^2}
   \norm{g(\lambda)}^2\cdot \abs{(h,f_*(\lambda_*))}^2\mathrm{d}\lambda
   \lesssim \norm{h}^2,
\end{equation}
\begin{equation}
\label{eq-est-main2}
\int_\Rset \delta^2(\lambda) \frac1{\abs{\varphi(\lambda)}^2}
   \norm{f_*(\lambda_*)}^2\cdot \abs{(h,g(\lambda))}^2\mathrm{d}\lambda
   \lesssim \norm{h}^2.
\end{equation}
\end{lemma}
\begin{proof}
Let us begin with \eqref{eq-est-main1}. We have
\ben
   \norm{R_\lambda(A)h} \ge
   \norm{\calL(\lambda){h}} - \norm{R_\lambda(B)h}
\]
whence
\begin{align*}
  \norm{\calL(\lambda)}^2
  & \le
  2\left[ \norm{R_\lambda(A)h}^2 + \norm{R_\lambda(B)h}^2\right].
\end{align*}
Thus
\ben
   \int_\Rset \delta^2(\lambda)\cdot\norm{\calL(\lambda){h}}^2
   \mathrm{d}\lambda
   \le
   C\norm{h}^2 + K\norm{h}^2,
\]
where we took into account \eqref{eq-res-est}
and  \eqref{eq-resB-est}.

The second inequality \eqref{eq-est-main2} is proved along the same lines.
\end{proof}

\section{PROOF OF THE THEOREM \ref{thm:neces}.}\lb{sec:proof}
Recall that items \eqref{en-cond1} -- \eqref{en-cond3} below
are tho\-se from this theorem.

\emph{First implication in the statement \eqref{en-cond1}}.
Combining
\eqref{eq-est-ginv} and \eqref{eq-est-main2},
we obtain the desired inequality
\begin{align*}
  \int_\Rset
   \frac{\abs{(g(\lambda),h)}^2}{\norm{g(\lambda)}^2}\mathrm{d}\lambda
  \lesssim
    \int_\Rset
  \delta^2(\lambda) \frac{1}{\abs{\varphi(\lambda)}^2}
   \norm{f_*(\lambda_*)}^2 \cdot
  \abs{(g(\lambda),h)}^2\mathrm{d}\lambda
  \lesssim
    \norm{h}^2.
\end{align*}
Next, we will need a lemma which is a counterpart of Lemma 4.1 in \cite{gub00}.
Its proof is essentially the same as in \cite{gub00}.
For the reader's sake we reproduce it here with necessary changes, putting
the factor $\delta(\lambda)$ in appropriate places.
\begin{lemma}
   \label{lem:w-int}
If $A\in\UB$, then
   the following estimates hold true
\begin{equation}
   \label{eq-w-int}
   \int_\Rset w^{\pm 2}(x)(1+x^2)^{-1}\mathrm{d}x < \infty.
\end{equation}
\end{lemma}
\begin{proof}
Convergence of the first integral stems from the identity
\cite[p.908]{gub00}
\be \lb{eq-iden1}
\int_\Rset x^{-2}\norm{g(x)-g(0)}^2\mathrm{d}x =
\int_\Rset \norm{R_x(B)g}^2\mathrm{d}x \le K\norm{g}^2.
\ee
Estimate the second integral.
From \eqref{eq-est-ginv} stems that
\be
  \label{eq-est-winv}
   \int_\Rset w^{-2}(x)(1+x^2)^{-1}\mathrm{d}x
   \lesssim
   \int_\Rset \delta^2(x)\abs{\varphi(x)}^{-2}
   \norm{f_*(x_*)}^2 (1+x^2)^{-1}\mathrm{d}x.
\ee
Following \cite[p.909-910 ]{gub00},
we put in \eqref{eq-est-main2}
$h = f_*(\mu_*)$ for
$\mu = \mu_1 \notin \sigma(A)$
and
$\mu = \mu_2 \in \sigma(A)$.
From \eqref{eq-biort-0} and \eqref{eq-biort-1} it follows that
\begin{align*}
&\int_\Rset \delta^2(\lambda)
  \abs{\varphi(\lambda)}^{-2}
   \norm{f_*(\lambda_*)}^2 \cdot
   \abs{\varphi(\lambda)-\varphi(\mu_1)}^2
   \abs{\lambda-\mu_1}^{-2}
   \mathrm{d}\lambda
   < \infty, \\
& \int_\Rset \delta^2(\lambda)
   \norm{f_*(\lambda_*)}^2 \cdot
   \abs{\lambda-\mu_2}^{-2}
   \mathrm{d}\lambda
   < \infty,
\end{align*}
whence
\ben
\int_\Rset \delta^2(\lambda)
  \abs{\varphi(\lambda)}^{-2}\cdot
   \norm{f_*(\lambda_*)}^2
   \abs{\lambda-\mu_1}^{-2}
   \mathrm{d}\lambda
   < \infty.
\]
Plugging this estimate into \eqref{eq-est-winv} yields the desired result.
\end{proof}

\emph{Second implication in the statement \eqref{en-cond1}}.
Recall that its counterpart is Proposition 4.2 in \cite{gub00}
whose
demonstration is based on the following assumptions:
\renewcommand{\labelenumi}{(\alph{enumi})}
\renewcommand{\theenumi}{\alph{enumi}}
\begin{enumerate}
	\item $B\in\xiexp$;        \lb{en-a}
	\item $g(\cdot)$ is \rreg; \lb{en-b}
	\item integrals from Lemma 4.1 in \cite{gub00} converge; \lb{en-c}
	\item validity of estimate (4.19) from Lemma 4.3 in \cite{gub00}. \lb{en-d}
\end{enumerate}
However \eqref{en-c} is established in our lemma \ref{lem:w-int}.
Next, \eqref{en-d} was proved in \cite{gub00} with reference to Lemma 4.2
in \cite{gub00} for $\eta=0$.
But the latter is exactly \emph{right-regularity} of $g(\cdot)$.
Thus we completed the proof of the second implication in \eqref{en-cond1}
as well as of the whole statement
\eqref{en-cond1}.

\emph{Statement \eqref{en-cond2}.}
It is enough to pass from $A\in\Ksigma$ to $A_*\in\Ksigma$ and apply
already established
statement \eqref{en-cond1}.

\emph{Statement \eqref{en-cond3}.}
It stems immediately from the relation
\be
 \lb{eq-main}
 \boxed{
        W^2(\lambda) \asymp w_*^{2}(\lambda), \ \lambda\in\Rset
    }
\ee
which is valid due to \eqref{eq-est-M}.

\thanks{I am indebted to my wife, Irina, for all her support,
and to D.Ya\-ku\-bo\-vich for help with literature.}

\end{document}